\documentclass{article}
\usepackage[normalem]{ulem}
\usepackage{amsthm}
\usepackage{mathtools}
\usepackage{thmtools}
\declaretheoremstyle[headfont=\normalfont]{normalhead}
\usepackage{color}
\usepackage{graphicx}
\usepackage{morefloats}
\usepackage{hyperref}
\usepackage{amssymb}
\usepackage{textcomp}
\usepackage[autostyle]{csquotes}
\usepackage{authblk}
\usepackage{dsfont}

\usepackage[nottoc]{tocbibind}

\usepackage{url}
\usepackage{float}

\newtheoremstyle{mydef}
{\topsep}{\topsep}%
{}{}%
{\itshape}{}
{\newline}
{
  \rule{\textwidth}{0.0pt}\\*
  \thmname{#1}~\thmnumber{#2}\thmnote{\-\ #3}.\\*[-1.5ex]
  \rule{\textwidth}{0.0pt}}

\newtheorem{theorem}{Theorem}

\newtheorem{remark}{Remark}
\newtheorem{proposition}{Proposition}

\newtheorem{lemma}{Lemma}

\newtheorem{observation}{Observation}
\newtheorem{definition}{Definition}
\newtheorem{example}{Example}
\newcommand{\dimb}{\text{dim}_B}

\newcommand{\R}{\mathbb{R}}
\newcommand{\N}{\mathbb{N}}

\numberwithin{equation}{section}
\mathtoolsset{showonlyrefs=true}

\author{Vlatko Crnković}

\affil{University of Zagreb, Faculty of Electrical Engineering and Computing, Department of Applied Mathematics, Unska 3, 10000 Zagreb, Croatia}
\affil{Hasselt University, Campus Diepenbeek, Agoralaan Gebouw D, 3590 Diepenbeek, Belgium}

\date{}

\title{Fractal Analysis of Pseudo Foci in Piecewise Analytic Systems}

\begin{document}

\maketitle
\begin{abstract}
\hskip -.2in
\noindent

It is well known that the Minkowski dimension of spiral trajectories near a non-degenerate focus in analytic (smooth) systems is in one-to-one correspondence with the cyclicity of the focus in generic unfoldings. We give a complete fractal treatment, in terms of the Minkowski dimension and (non-)degeneracy, of spiral trajectories near \emph{pseudo foci} of piecewise analytic systems. We hope these results will prove useful in the study of bifurcations of such pseudo foci through inherent geometry of associated spiral orbits.

\end{abstract}

\tableofcontents

\section{Introduction}\label{sec:intro}

\subsection{Fractal analysis of non-degenerate foci in analytic systems}

    The Minkowski dimension is a fractal dimension that quantifies how the Lebesgue measure of the $\delta$-neighbourhood of a bounded set in $\mathbb{R}^N$ behaves as $\delta \to 0$.
    There are several equivalent ways of calculating the value of the Minkowski dimension of a set, but for our purposes we will use the following one
    
    \begin{definition}[Minkowski dimension, \cite{tricot95}]
        For a bounded set $G\subset \mathbb{R}^N$, we call the set
        \[ G_\delta \colon = \{ p \in \mathbb{R}^N \colon \text{dist}(p, G) < \delta\},\quad \delta > 0, \]
        be the $\delta$-neighbourhood of $G$ and denote its Lebesgue measure by $|G_\delta|$.\\
        If the limit 
        \begin{equation}
            \lim_{\delta \to 0+} \left[ N - \frac{\ln |G_\delta|}{\ln \delta} \right]
        \end{equation}
        exists, we call its value the Minkowski dimension of $G$, and denote it by $\dimb G$. Moreover, if 
        \begin{equation}
            0 < \liminf_{\delta \to 0+} \frac{|G_\delta|}{\delta^{N-\dimb G}} \leq \limsup_{\delta \to 0+} \frac{|G_\delta|}{\delta^{N-\dimb G}} < + \infty,
        \end{equation}
        we say that $G$ is Minkowski non-degenerate.
    \end{definition}
    The Minkowski dimension is preserved under bi-Lipschitz transformations, even when the image and the original are not in the same ambient space. More precisely, a transformation $\Psi: A \subset \mathbb{R}^N \to \mathbb{R}^K$ is called bi-Lipschitz if there exit positive constants $m$ and $M$ such that for any $x,y \in A$ 
   \begin{equation}
        m || \Psi(x) - \Psi(y) || \leq ||x-y|| \leq M || \Psi(x) - \Psi(y)||,
    \end{equation}
    and for such $\Psi$ we have that
    \begin{equation}
        \dimb A = \dimb \Psi (A).
    \end{equation}
    The dimension is also monotone in the sense that for $G\subseteq H \implies \dimb G \leq \dimb H$ when both are defined.
    Additionally, it is finitely stable, meaning that $\dimb {G \cup H} = \max \{ \dimb G, \dimb H\}$.
    For more on these and other properties of the Minkowski dimension we refer the reader to \cite{falconer90} and \cite{tricot95}.

    It is already well known that the Minkowski dimension of spirals winding around weak foci of planar analytic vector fields yields information on the cyclicity of those limit periodic sets.
    The following is the first result of this type.
    \begin{theorem}[Žubrinić, Županović, 2005, \cite{zz05}]\label{tm:focus_dim}
        Let $\Gamma$ be a part of a trajectory of the system
        \[ \begin{cases}
        \dot{r} = r(r^{2l} + \sum_{i=0}^{l-1} a_ir^{2i})\\
        \dot{\theta} = 1 \end{cases} \]
        near the origin.
        Then
        \begin{enumerate}
            \item[(a)] if $a_0 \neq 0$, then $\dimb \Gamma = 1$.
            \item[(b)] if $a_0 = a_1 = ... = a_{k-1} = 0, a_k \neq 0$, then $\dimb \Gamma = \frac{4k}{2k+1}$.
        \end{enumerate}
    \end{theorem}
    
    The conclusions of this theorem, the one-to-one correspondence between the order of the focus and the Minkowski dimension of spiral trajectories near it, hold for all non-degenerate foci in analytic (smooth) systems. This is evident from later work of the authors of the theorem, where they have developed the \emph{flow-sector} theorem.
    \begin{theorem}[The flow-sector theorem, \cite{zz08}]
        Consider the system 
                \begin{equation}\label{eq2.5}
                    \begin{cases}
                        \dot{x} = -y + m(x,y),\\
                        \dot{y} = x + n(x,y)
                    \end{cases}
                \end{equation}
                where $m(x,y), n(x,y) = O(x^2+y^2)$ as $(x,y)\to (0,0)$ are $C^1$ functions. Let $U_0\subset \mathbb{R}^2$ be an open sector with the vertex at the origin, such that its opening angle is in $(0, 2\pi)$, and the boundary of $U_0$ consists of a part of a trajectory of \eqref{eq2.5} and of intervals on two rays emanating form the origin. If the diameter of $U_0$ is sufficiently small, then system \eqref{eq2.5} is lipeomorphically equivalent to the system
                \begin{equation}\label{eq2.6}
                    \begin{cases}
                        \dot{r} = 0,\\
                        \dot{\theta} = 1
                    \end{cases}
                \end{equation}
                defined on the sector $V_0 = \left\{ (r,\theta) \colon 0<r<1,\, 0<\theta < \frac{\pi}{2}\right\}$ in polar coordinates $(r,\theta)$.
    \end{theorem}

    \begin{figure}[h!]
        \centering
        \includegraphics[width=0.5\linewidth]{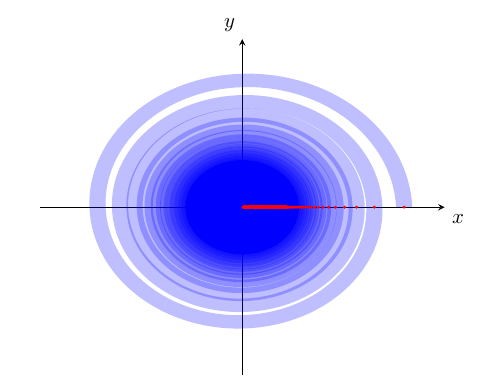}
        \caption{The $\delta$-neighbourhood of a spiral trajectory in \autoref{tm:focus_dim}, \cite{thesis_Crnkovic}}
        \label{fig:focus_saussage}
    \end{figure}
    
    The flow-sector theorem, together with the finite stability property of the Minkowski dimension, allows one to reduce the analysis of spiral trajectories to the analysis of orbits of the first return map associated to the focus. Finally, the following result provides all relevant information on the orbits of analytic first return maps.
    \begin{theorem}[Fractal analysis of line diffeomorphisms, \cite{ezz07}]\label{tm:neveda}
        Let $\alpha > 1$ and let $f\colon (0, r) \to (0, +\infty)$ be a monotonically nondecreasing function such that $f(x)\simeq x^\alpha$ as $x\to 0$, and $f(x) < x$ for all $x\in (0,r)$. Consider the sequence $S(x_0)\vcentcolon = (x_n)_{n\geq 0}$ defined by
        \begin{equation*}
            x_{n+1} = x_n - f(x_n),\quad x_0\in (0,r).
        \end{equation*}
        Then
        \begin{equation*}
            x_n\simeq n^{-\frac{1}{\alpha-1}},\quad \text{as } n\to \infty.
        \end{equation*}
        Furthermore,
        \begin{equation*}
            \dimb\, S(x_0) = \frac{1}{\frac{1}{\alpha-1}+1} = 1 - \frac{1}{\alpha}
        \end{equation*}
        and the set $S(x_0)$ is Minkowski non-degenerate.
    \end{theorem}

    \subsection{Pseudo foci in piecewise analytic systems}

    In this paper we consider the following model for piecewise analytic systems. Let $V^\pm = M^\pm \partial_x + N^\pm \partial_y$ be analytic vector fields in an open neighbourhood of the origin in $\mathbb{R}^2$. If $N^+(x,0)N^-(x,0)>0$ for $x\neq 0$, one naturally defines the flow of the piecewise analytic vector field $V = \mathds{1}_{\{y\geq0\}}V^+ + \mathds{1}_{\{y\leq0\}}V^-$ to agree with the flow of $V^+$ in the upper half plane $\{y>0\}$, with the flow of $V^-$ in the lower half plane $\{y<0\}$, and to switch between them as it crosses the $x-$axis. Depending of the local behaviour of $V^+$ and $V^-$ at the origin, this can result in the origin being a focus-like fixed point of the flow of $V$. We study the following three types of such \emph{pseudo foci}, as defined in \cite{cgp01}. For the sake of simplicity, we will assume that the resulting flow follows a counter-clockwise direction around the origin, which is equivalent to $\mathrm{sgn\ } N^+(x,0) = \mathrm{sgn\ } N^-(x,0) = \mathrm{sgn\ } x$.

    \subsubsection*{FF type pseudo foci}

        If both $V^+$ and $V^-$ have a non-degenerate focus at the origin, we say that the origin is a pseudo focus of $V$ of the FF (focus-focus) type.

        \begin{figure}[h!]
            \centering
            \includegraphics[width=0.5\linewidth]{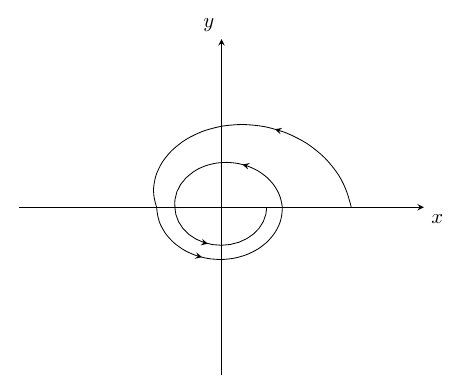}
            \caption{A pseudo focus of the FF type, \cite{thesis_Crnkovic}}
        \end{figure}

    \subsubsection*{PP type pseudo foci}

        If both $V^+$ and $V^-$ have a second-order contact with the line $\{y = 0\}$ at the origin, which is equivalent to \begin{equation}\label{eq:par1}
            N^\pm(0,0) = 0,
        \end{equation}
        and
        \begin{equation}\label{eq:par2}
            M^\pm(0,0)\cdot\partial_x N^\pm(0,0) > 0,
        \end{equation}
        we say that the origin is a pseudo focus of $V$ of the PP (parabolic-parabolic) type.

        \begin{figure}[h!]
            \centering
            \includegraphics[width=0.5\linewidth]{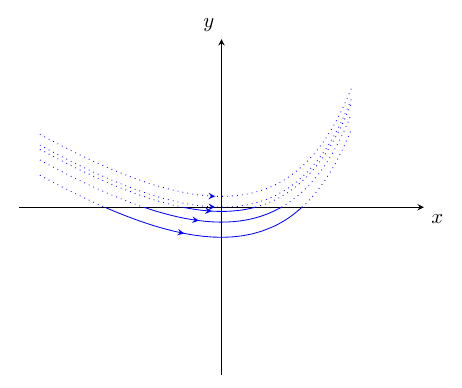}
            \caption{Parabolic contact with $\{y = 0\}$, \cite{thesis_Crnkovic}}
            \label{fig:placeholder}
        \end{figure}

        \begin{figure}[h!]
            \centering
            \includegraphics[width=0.5\linewidth]{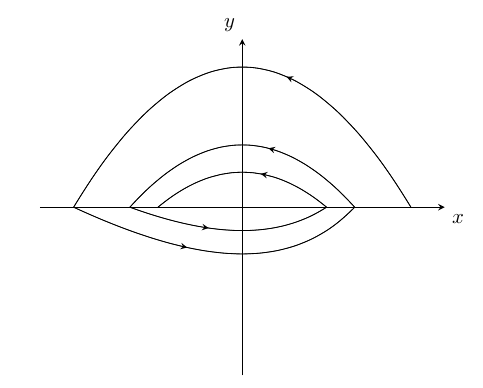}
            \caption{A pseudo focus of the PP type, \cite{thesis_Crnkovic}}
        \end{figure}

    \subsubsection*{FP/PF type pseudo foci}

        If $V^+$ (resp. $V^-$) has a non-degenerate focus at the origin, and $V^-$ (resp. $V^+$) has a second order contact with $\{y = 0\}$ at the origin, we say that the origin is a pseudo focus of $V$ of the FP (resp. PF) type. We will refer to both of these types of pseudo foci as pseudo foci of \emph{mixed} type.

        \begin{figure}[h!]
            \centering
            \includegraphics[width=0.5\linewidth]{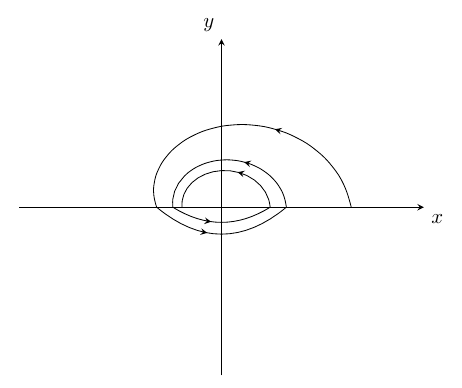}
            \caption{A pseudo focus of FP type, \cite{thesis_Crnkovic}}
        \end{figure}

    \subsubsection{Cyclicity of pseudo-foci}

        In their pivotal work, \cite{cgp01}, the authors have found ways to calculate some \emph{Lyapunov quantities} (coefficients in the asymptotic expansion of the first return map) for pseudo-foci of such systems. Similarly as in the the analytic (smooth) case, these quantities carry information on the cyclicity of (pseudo-)foci. The exploration of this connection was continued through the work of several researchers over the last twenty years, for instance \cite{lh12, crz15} and many more. Our fractal methods will allow us to find the correspondence between the so-called \emph{order} of the pseudo-focus, and the fractal properties of nearby orbits, as was done already in analytic (smooth) case, thus showing that fractal properties are a suitable proxy for relevant information on cyclicity first return maps carry.

    \subsection{Main results of the paper}

    \subsubsection{Fractal analysis of pseudo foci}
        
        As in the case of foci in analytic systems, first return maps associated to these pseudo foci are also analytic (see \autoref{prop:para_trans_char}). If the first return map is of the form $P(x) = x + \alpha_k x^k + o(x^k)$, $k\in \mathbb{N}$, $\alpha_k\neq 0$, we say that the pseudo focus is of \emph{order} $k$. The main results of this paper give a complete characterization of the order of the pseudo focus in terms of the fractal properties (Minkowski dimension and (non-)degeneracy) of the associated spiral trajectories.
    
        \begin{theorem}[Fractal properties of pseudo foci of PP type, \cite{thesis_Crnkovic}]\label{thm:PPfrac}
            Consider a PP type pseudo focus of order $k$ of a piecewise analytic vector field. Then $k$ is necessarily even and the Minkowski dimension of any spiral trajectory near the focus is $2-\frac{3}{k+1}$. The spiral in question is Minkowski non-degenerate if and only if $k\neq 2$.
        \end{theorem}
    
        \begin{theorem}[Fractal properties of pseudo foci of FF and mixed type, \cite{thesis_Crnkovic}]\label{thm:FFfrac}
            For a PF/FP or an FF type pseudo focus of order $k$ of a piecewise analytic vector field, the Minkowski dimension of any spiral trajectory near the focus is $2-\frac{2}{k}$ if $k>1$, and $1$ if $k = 1$. The spiral in question is Minkowski non-degenerate if and only if $k\neq 2$.
        \end{theorem}

    \subsubsection{The reconstructibility question}

        A natural question to pose is if the fractal properties of a focus/foci of $V^+$ and/or $V^-$ can be reconstructed from the fractal properties of the pseudo focus of $V$. We call this the \emph{reconstructibility question}. An indirect answer to this question is given in the form of the following theorems.

        \begin{theorem}[Realization of FF type pseudo foci, \cite{thesis_Crnkovic}]\label{thm:FF_real}
            For any $k_1\geq  k_2\geq 0$ the following holds:
            \begin{enumerate}
                \item If $k_1 = k_2$, for any $k\in \mathbb{N}$ the following holds:
                There exist analytic vector fields $V_1$ and $V_2$, respectively with foci of order $k_1$ and $k_2$ at the origin, such that the piecewise analytic vector field $V = \mathds{1}_{\{y>0\}}V_1 + \mathds{1}_{\{y<0\}}V_2$ has a pseudo focus of order $k$ at the origin if and only if $k\geq 2k_1 + 1=2k_2 + 1$ or $k$ is even.

                \item If $k_1 > k_2$, for any $k\in \mathbb{N}$ the following holds:
                There exist analytic vector fields $V_1$ and $V_2$, respectively with foci of order $k_1$ and $k_2$ at the origin, such that the piecewise analytic vector field $V = \mathds{1}_{\{y>0\}}V_1 + \mathds{1}_{\{y<0\}}V_2$ has a pseudo focus of order $k$ at the origin if and only if $k = 2k_2 + 1$; or $k$ is even and smaller than $2k_2 + 1$.
            \end{enumerate}
        \end{theorem}

        \begin{theorem}[Realization of mixed type pseudo foci, \cite{thesis_Crnkovic}]\label{thm:mixed_real}
            For any $k\in \mathbb{N}_0$, $n\in \mathbb{N}$, the following holds: there exist analytic vector fields $V_1$ with a focus of order $k$ at the origin and $V_2$ with a parabolic contact or center at the origin such that the piecewise analytic vector field $V = \mathds{1}_{\{y>0\}}V_1 + \mathds{1}_{\{y<0\}}V_2$ has a pseudo focus of order $n$ at the origin if and only if $n = 2k+1$; or $n$ is even and smaller than $2k+1$.
        \end{theorem}

        \begin{remark}
            As defined, the types of pseudo foci we study consist only of weak or hyperbolic foci and parabolic contacts. However, any transition map a system with a parabolic contact exhibits can also be obtained as a semi-monodromy map in a system with a monodromic center (see \autoref{prop:focus_semi_real} and \autoref{prop:para_trans_char}). Therefore, \autoref{thm:mixed_real} is stated for systems that can also have monodromic centers instead of parabolic contacts.
        \end{remark}

    \subsection{Conventions and notations}

    We use two notions of equivalent asymptotic behaviour of functions with small arguments. Let $I\subset \mathbb{R}$ be an open interval either containing $0$, or having $0$ as one of its ends. For two functions $f,g\vcentcolon I \to \mathbb{R}$ that don't vanish except maybe at $0$, we write $f(x)\simeq g(x)$, as $x\to 0$, if there exist positive constants $m$ and $M$ such that
    \begin{equation}
        m|f(x)|\leq g(x) \leq M|f(x)|,\quad x\in I.
    \end{equation}
    Moreover, we write $f(x)\sim g(x)$, as $x\to 0$, if
    \begin{equation}
        \lim_{x\to 0} \frac{f(x)}{g(x)} = 1.
    \end{equation}

\section{First return maps}\label{sec:first-return}

    For a complete understanding of pseudo foci, one would like to characterize maps that can be realized as first return maps of such fixed points. Unfortunately, we are not able to obtain a result of such strength, but it is not hard to obtain a weaker version of the results that will be sufficient for our further fractal considerations. Let's first recall the structure of the first return maps associated to non-degenerate foci in analytic systems.

    \subsection{First return maps of non-degenerate foci}\label{ssec:fr-nondeg}

       Let the origin be a non-degenerate monodromic singularity (associated eigenvalues have non-zero imaginary part) of an analytic vector field, and let $\gamma\vcentcolon\langle -\varepsilon', \varepsilon'\rangle \to \mathbb{R}^2,\, \varepsilon' > 0$, be an analytic parametrization of an arbitrary line segment through the origin, such that $\gamma(0) = (0,0)$. Then there is an $\varepsilon \in \langle 0, \varepsilon'\rangle$, and an analytic map $h\vcentcolon \langle -\varepsilon, \varepsilon\rangle\to \langle -\varepsilon', \varepsilon'\rangle$, satisfying $h'(0) < 0$, such that $\gamma(h(t))$ is the first point where a forward trajectory through $\gamma(t)$ intersects the image of $\gamma$. Moreover, $P \vcentcolon = h\circ h$ represents the first return to $\gamma\left(\langle 0, \varepsilon'\rangle\right)$ and to $\gamma\left(\langle - \varepsilon', 0\rangle\right)$. A standard result (see for instance \cite{roussarie98}) is that either $P(t)\equiv t$, or that the associated displacement map is of the form
        \begin{equation}
            \Delta(t) \vcentcolon = P(t) - t = \alpha_{2k+1}t^{2k+1} + o(t^{2k+1}),
        \end{equation}
        where $t\in \mathbb{N}_0$ and $\alpha_{2k+1}\neq 0$. In the former case the singularity is a center, and in the latter case, we say that the singularity is a focus of \emph{order} $k$. We refer to $P$ as the \emph{monodromy map}, and to $h$ as the \emph{semi-monodromy map} associated to a non-degenerate monodromic singularity.

    \subsection{Semi-monodromy maps of non-degenerate foci}\label{ssec:sm-nondeg}

        In the study of foci in analytic systems, one is usually interested in the monodromy map. However, since in piecewise smooth systems we study, the transitions through the upper and lower half-plane define different "semi-monodromy" maps, we will restrict our attention to these half-turns. We have seen in \autoref{ssec:fr-nondeg} that the semi-monodromy associated to a non-degenerate monodromic singularity is analytic in any analytic parametrization of the transversal. A natural question to pose is "what are all the analytic germs that can be realized as such semi-mondromy maps?". Unfortunately, we don't have an answer to this question, but we obtain the following weaker version of the result. It will be sufficient for our considerations.

        \begin{proposition}\label{prop:focus_semi_real}
            Let $h(t) = -\lambda t + o(t)$, $\lambda > 0$, be a real analytic germ at $0$. Then there is an analytic vector field with a non-degenerate monodromic singularity at the origin, such that the associated semi-monodromy map $m$ is a representative of the formal class of $h$. Moreover, for an arbitrary $N\in\mathbb{N}$, such a vector field can be chosen so that $h(t) - m(t) = o(t^N)$.
        \end{proposition}

        \begin{remark}
            Notice that the realization of $h$ is equivalent to the realization of any representative $m$ of its analytic class, that is any $m$ such that there is a local diffeomorphism $\varphi$, analytic at $0$, satisfying $h = \varphi\circ m \circ \varphi^{-1}$. Indeed, if $m$ is the semi-monodromy of some analytic vector field $V$, then $\varphi\circ m\circ \varphi^{-1}$ is the semi-monodromy of the vector field $\Phi_* V$, where $\Phi(x,y)\vcentcolon = \left( \varphi(x), y\right)$.
        \end{remark}

        Since the determination of an analytic class is a hard problem, we work with formal classes of maps.

        \begin{definition}
            Let $f(t) = \mu t + o(t)$ and $g(t) = \nu t + O(t)$, $\mu, \nu \neq 0$, be analytic germs at $0$. We say that $f$ and $g$ are \emph{formally equivalent} if for any $N\in \mathbb{N}$, there is a local diffeomorphism $\varphi(t) = t + o(t)$, analytic at $0$, such that
            \begin{equation}
                \varphi \circ f \circ \varphi^{-1} (t) = g(t) + o(t^N).
            \end{equation}
        \end{definition}

        Unlike the usual notion of formal equivalence (see for instance \cite{abbdss08}), we restrict our conjugacy maps to be tangent to the identity. This also reduces the formal class, but makes the following calculations more manageable.

        The formal class of a germ can be known from only finitely many terms in the Taylor expansion. This is precisely stated in the form of following lemmas. Analogous statements for complex germs $(\mathbb{C}, 0)\to (\mathbb{C}, 0)$ can be found in \cite{abbdss08}.

        \begin{lemma}[Formal class of hyperbolic germs]\label{lm:formal_hyp}
            A hyperbolic germ $f(t) = \lambda t + o(t)$, $\lambda \neq 0,\pm 1$, is formally equivalent to its linearization $t\mapsto \lambda t$.
        \end{lemma}

        \begin{lemma}[Formal class of orientation preserving parabolic germs]\label{lm:formal_para_pres}
            Let $f(t) = t + \alpha_kt^k + o(t^k)$, $k\geq 2, \alpha_k\neq 0$, be a parabolic germ. The formal class of $f$ is completely determined by its $(2k-1)$-jet. More precisely, $k$ and $\alpha_k$ are invariants of the formal class, and there is exists a unique real $\beta$ such that $f$ is formally equivalent to $t\mapsto t + \alpha_kt^k + \beta t^{2k-1}$.
        \end{lemma}

        \begin{lemma}[Formal class of orientation reversing parabolic germs]\label{lm:formal_para_rev}
            Let $f(t) = -t + o(t)$ be analytic at $0$. Then $f$ is either an involution, and thus analytically conjugate to $t\mapsto -t$, or there exist unique $k\geq 1$, $a_{2k+1}\neq 0$ and $ a_{4k+1}\in \mathbb{R}$ such that $f\circ f$ is formally equivalent to 
            \begin{equation*}
                t\mapsto t + a_{2k+1}t^{2k+1} + a_{4k+1}t^{4k+1}.
            \end{equation*}
            In the latter case, if $f(t) = -t + ct^m + o(t^m)$ for $c\neq 0$, then $m \leq 2k+1$ and $m$ is even if and only if $m \neq 2k+1$. Furthermore, in this case $f$ is formally equivalent to
            \begin{equation}\label{eq4.41}
                t\mapsto -t - \frac{a_{2k+1}}{2}t^{2l+1} + \frac{(2k+1)a_{2k+1}^2 - 4a_{4k+1}}{8} t^{4k+1}.
            \end{equation}
        \end{lemma}

        The proofs of these lemmas are fairly straightforwardly obtained by consecutive eliminations of terms in the Taylor expansions. The only non-obvious step is the conclusion in \autoref{lm:formal_para_rev} that the first non-linear term in the Taylor expansion of $f\circ f$ has to correspond to an odd power. This becomes evident when one considers the relative position of $f(t)$ and $f^{\circ 3}(t)$, for some sufficiently small positive $t$.

        We are now ready to prove \autoref{prop:focus_semi_real}.

        \begin{proof}[Proof of \autoref{prop:focus_semi_real}]
            The proof is based on \autoref{lm:formal_hyp}, \autoref{lm:formal_para_rev}, and the following constructions. We omit the details.
            
            For $\lambda \in \mathbb{R}$, $t\mapsto -\exp\left(\lambda\right) t$ is the semi-monodromy map of the system 
            \begin{equation}
                \begin{cases}
                    \dot{x} = \lambda x - \pi y,\\
                    \dot{y} = \pi x + \lambda y.
                \end{cases}
            \end{equation}

            For $k \in \mathbb{N}$, $\alpha \neq 0$ and $\beta \in \mathbb{R}$, the semi-monodromy of the system 
            \begin{equation}
                \begin{cases}
                    \dot{x} = -\pi y - x\left(\alpha (x^2+y^2)^k + \gamma (x^2+y^2)^{2k}\right),\\
                    \dot{y} = \pi x - y\left(\alpha (x^2+y^2)^k + \gamma (x^2+y^2)^{2k}\right),
                \end{cases}
            \end{equation}
            where $\gamma = \beta + \frac{(2k+1)\alpha^2}{2}$, is of the form $t\mapsto -t + \alpha t^{2k+1} + \beta t^{4k+1} + o(t^{4k+1})$.
        \end{proof}

    \subsection{Transition maps near parabolic contacts}

        As we've mentioned already, first return maps near pseudo foci will be compositions of semi-monodromy maps of foci and/or \emph{transition} maps associated to the first return to the $x$-axis in systems with a parabolic contact. The following proposition completely characterizes such transitions.

        \begin{proposition}\label{prop:para_trans_char}
            Let 
            \begin{equation}
                \begin{cases}\label{eq:para_contact1}
                    \dot{x} = M(x,y),\\
                    \dot{y} = N(x,y)
                \end{cases}
            \end{equation}
            be a system with a parabolic contact with the line $\{y = 0\}$ at the origin. Then for any $x$ sufficiently close to $0$, the trajectory through $(x,0)$ locally intersects the $x$-axis once again at a unique point $(h(x), 0)$, such that $h$ is analytic at $0$ and an involution. Moreover, for any involution $h$ analytic at $0$, there is an analytic system with a parabolic contact with $\{y = 0\}$ at the origin such that $h$ is the associated transition map.

        \begin{figure}[h!]
            \centering
            \includegraphics[width=0.5\linewidth]{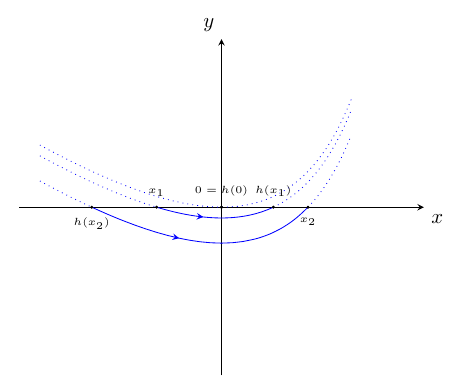}
            \caption{The transition map associated to a parabolic contact, \cite{thesis_Crnkovic}}
            \label{fig:para_trans}
        \end{figure}

        \begin{proof}
            Without loss of generality, we assume that $M\equiv 1$. For $x$ and $t$ sufficiently close to $0$, we define 
            \begin{equation}
                H(x,t) = \frac{Y(t,x,0)}{t},
            \end{equation} 
            where $Y(t, x,y)$ is the $y$-component of the $t$-flow of \eqref{eq:para_contact1}, and extend it by continuity to $\{t = 0\}$, where $H(x,0) = N(x,0)$. Since
            \begin{equation}
                \partial_t H(0,0) = \frac{\partial_x N(0,0)}{2} \neq 0,
            \end{equation}
            the implicit function theorem guarantees the existence of a unique analytic map $x\mapsto t(x)$, such that $H(x, t(x)) = 0$. Since $M\equiv 1$, $x + t(x)$ is the $x$-coordinate of the $t(x)$-flow of $(x,0)$. Therefore, $h(x) \vcentcolon = x + t(x)$ is the transition map in question. Since, 
            \begin{align*}
                0 & = Y(0, x, 0)\\
                & = Y(-t(x) + t(x), x, 0)\\
                & = Y(-t(x), X(t(x), x, 0), Y(t(x), x, 0))\\
                & = Y(-t(x), x+t(x), 0)\\
                & = Y(-t(x), h(x), 0),
            \end{align*}
            it is clear that
            \begin{equation}
                h(h(x)) = h(x) + (-t(x)) = x + t(x) - t(x) = x, 
            \end{equation}
            i.e. $h$ is an involution.

            If $h$ is an arbitrary involution analytic at $0$, then the system
            \begin{equation}
                \begin{cases}\label{eq:para_hamil}
                    \dot{x} = 1,\\
                    \dot{y} = h(x) + xh'(x),
                \end{cases}
            \end{equation}
            is Hamiltonian with the potential $F(x,y) = y - xh(x)$. Since
            \begin{align*}
                F(h(x), 0) = - h(x)h(h(x)) = - xh(x) = F(x, 0),
            \end{align*}
            it is clear that $h$ is the transition map associated to the system \eqref{eq:para_hamil}.
        \end{proof}
        \end{proposition}

\section{Fractal analysis of pseudo foci}\label{sec:analysis}

\subsection{Separation of halfplanes and the reduction of \autoref{thm:PPfrac} and \autoref{thm:FFfrac}}
Fractal analysis of spiral trajectories in \autoref{thm:PPfrac} and \autoref{thm:FFfrac} is done separately in the upper and lower half plane. In the half plane where a focus is present, a variation of the so-called \emph{flow-sector theorem}, \cite{zz08}, is used. In the half plane where the system is governed by the parabolic contact, we reduce the analysis to that of geometric chirp-like sets (see \cite{lr24}). Similarly as in the analytic case, these analyses allow us to study only the orbits of the first return maps.

\begin{theorem}[Hyperbolic flow-sector theorem for analytic systems, \cite{zz08, thesis_Crnkovic}]\label{tm:flow_sec}
    Let 
    \begin{equation}\label{eq4.10}
    \begin{cases}
        \dot{x} = \lambda x - y + m(x,y),\\
        \dot{y} = x + \lambda y + n(x,y),
    \end{cases}
    \end{equation}
    with $\lambda \in \mathbb{R}$ and $m, n = O(x^2+y^2)$ be an analytic system. Then for a small enough open sector $U_0\subset \mathbb{R}^2$ with a vertex at the origin, such that its opening angle is in $(0, \pi]$, and boundary consist of a part of a trajectory of \eqref{eq4.10} and of intervals on two rays emanating from the origin, there is a bi-Lipschitz radial function $\varphi\vcentcolon \overline{U_0} \to \mathbb{R}^2$ that is acting as the identity on one of the ray segments on $\partial U_0$, such that under the change of coordinates $\varphi$ the system \eqref{eq4.10} on $\overline{U_0}$ becomes orbitally the same as
    \begin{equation}\label{eq4.11}
        \begin{cases}
        \dot{x} = - y,\\
        \dot{y} = x,
    \end{cases}
    \end{equation}
    on a sector of opening angle $\theta_0$ and a vertex at the origin.
    \begin{proof}
        A slight modification of the proof of flow sector theorem, \cite[Theorem 3]{zz08}. A radial transformation of coordinates on $U_0$ eliminates $\lambda$, and we are more or less left with the same situation as in the (classical) flow-sector theorem. More details can be found in \cite{thesis_Crnkovic}.
    \end{proof}
\end{theorem}

The flow sector theorem allows us to reduce the studied configuration to the family of concentric circle arcs, which simplifies the analysis. Similarly, in the case of the parabolic contact, the following observation allows us to reduce the configuration to the family of parallel line segments with endpoints on a parabola-like curve.

\begin{observation}\label{ob:para}
    Since we are only interested in the orbits of a system with a parabolic contact, without loss of generality we may assume that it is of the form
    \begin{equation}\label{eq:para_norm}
        \begin{cases}
            \dot{x} = 1,\\
            \dot{y} = N(x,y) = \alpha x + O(y) + O(x^2),
        \end{cases}
    \end{equation}
    where $N$ is analytic at the origin and $\alpha>0$. Let $Y(t, x, y)$ denote the $y$-component of the $t$-flow of \eqref{eq:para_norm} starting at $(x,y)$. Under the change of coordinates
    \begin{equation*}
        (x,y)\mapsto (x, Y(-x, x, y)),
    \end{equation*}
    the system \eqref{eq:para_norm} trivializes to
    \begin{equation}
        \begin{cases}
            \dot{x} = 1,\\
            \dot{y} = 0,
        \end{cases}
    \end{equation}
    and the integral curves become lines $\{ y = const.\}$. The $x$-axis is mapped to the curve $\{ y = Y(-x, x, 0)$, where $x\mapsto Y(-x, x, 0) = -\frac\alpha 2 x^2 + O(x^3)$ is analytic. See \autoref{fig:para_straight}.

    \begin{figure}
        \centering
        \includegraphics[width=\linewidth]{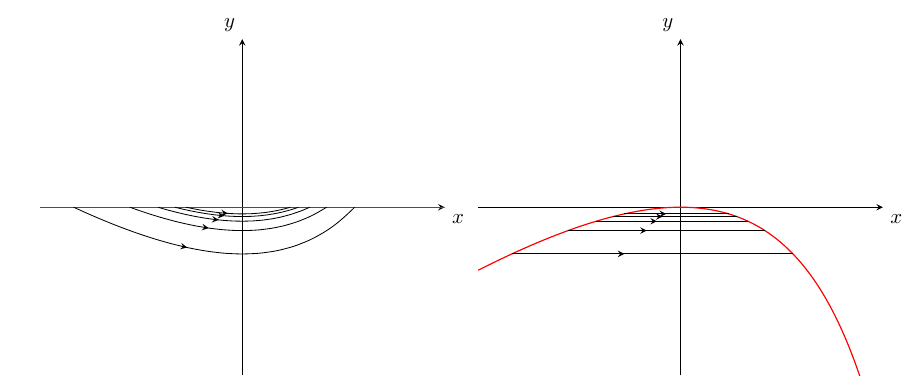}
        \caption{"Straightening" of a parabolic contact, \cite{thesis_Crnkovic}}
        \label{fig:para_straight}
    \end{figure}
\end{observation}

These reductions gives us the following fractal results.

\begin{proposition}\label{prop:dim_focus}
    Consider the system
    \begin{equation}\label{eq4.25}
        \begin{cases}
            \dot{x} = \lambda x - y + m(x,y),\\
            \dot{y} = x + \lambda y + n(x,y),
        \end{cases}
    \end{equation}
    where $\lambda\in\R$ and $m,n = O(x^2+y^2)$ are analytic at the origin. Consider also a function $P(x) = x - cx^k + o(x^k)$, $k\in\N, c>0$, analytic at $0$, where $c<1$ if $k=1$. For $x_0>0$ sufficiently small, define $x_n \vcentcolon= P^n(x_0)$, and let $\Gamma_n$ be parts of trajectories of \eqref{eq4.25} between $(x_n, 0)$ and $\{x<0\}\times\{0\}$. Then, for the Minkowski dimension of $\cup_{n\in\N_0} \Gamma_n$ we have
    \begin{equation*}
        \dimb \cup_{n\in\N_0} \Gamma_n = \begin{cases}
            2 - \frac{2}{k},& k \geq 2,\\
            1,& k = 1.
        \end{cases}
    \end{equation*}
    Moreover, $\cup_{n\in\N_0} \Gamma_n$ is Minkowski non-degenerate if and only if $k\neq 2$.
\end{proposition}

    \begin{figure}[h!]
        \centering
        \includegraphics[width=0.55\linewidth]{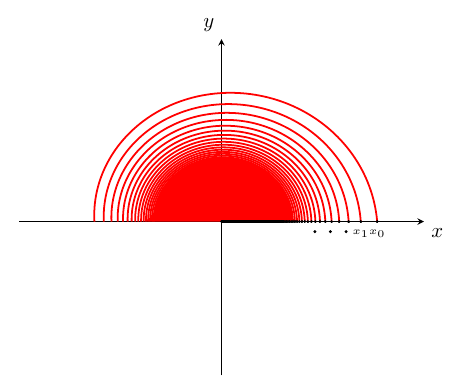}
        \caption{Family $\left(\Gamma_n\right)_{n\in\N}$ of trajectories in \autoref{prop:dim_focus}, \cite{thesis_Crnkovic}}
        \label{fig:focus_analysis}
    \end{figure}

\begin{proposition}\label{prop:dim_para}
    Assume that the system
    \begin{equation}\label{eq:gen-parab}
        \begin{cases}
            \dot{x} = M(x,y),\\
            \dot{y} = N(x,y),
        \end{cases}
    \end{equation}
    has a parabolic contact with the line $\{y = 0\}$ at the origin, and let the transition between the positive and negative part of the $x$-axis be represented by an involution $h(x) = -x + O(x^2)$ analytic at $0$ (as explained in \autoref{prop:para_trans_char}). Let $P(x) = x -cx^k + o(x^k)$, $k\in \N$, $c>0$, be analytic at $0$, where $c<1$ if $k=1$. For $x_0>0$ sufficiently small, define $x_n \vcentcolon= P^n(x_0)$, and let $\Gamma_n$ be parts of trajectories of \eqref{eq:gen-parab} between $(h(x_n), 0)$ and $(x_n, 0)$. Then
    \begin{equation*}
        \dimb \cup_{n\in\N_0} \Gamma_n = \begin{cases}
            2 - \frac{3}{k+1}, & k\geq 2,\\
            1, & k = 1.
        \end{cases}
    \end{equation*}
    Moreover, $\cup_{n\in\N_0} \Gamma_n$ is Minkowski non-degenerate if and only if $k\neq 2$.
\end{proposition}

\begin{figure}[h!]
    \centering
    \includegraphics[width=0.55\linewidth]{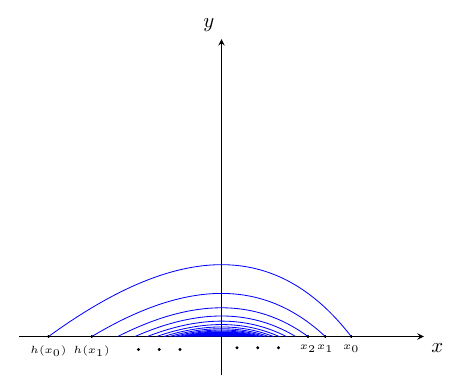}
    \caption{Family $\left(\Gamma_n\right)_{n\in\N}$ of trajectories in \autoref{prop:dim_para}}
    \label{fig:pre_st_chirp}
\end{figure}

\begin{remark}
    Notice that \autoref{thm:PPfrac} and \autoref{thm:FFfrac} are direct consequences of \autoref{prop:dim_focus}, \autoref{prop:dim_para}, and the \emph{finite stability} of the Minkowski dimension.
\end{remark}

\subsection{Proofs of \autoref{prop:dim_focus} and \autoref{prop:dim_para}}

Before the proofs of these propositions, let's demonstrate how one can calculate the Minkowski dimension of some sequences of real numbers.

\begin{example}[\cite{tricot95}]\label{exa:nuc_tail}
    For $\alpha > 0$, consider the sequence of points $x_n = \frac{1}{n^\alpha},\, n\in \N$. Let $\delta > 0$ be arbitrary. Since the distance between consecutive elements of the sequence $(x_n)_{n\in\N}$ is decreasing, there is a unique index $n_\delta$, such that $x_n - x_{n+1} \geq 2\delta,\, n < n_\delta$ and $x_n - x_{n+1} < 2\delta,\, n\geq n_\delta$. Now the $\delta$-neighbourhood of $\{x_n\vcentcolon n\in \N\}$ is a disjoint union of sets
    \begin{equation}
        N_\delta \vcentcolon= \{x_n\vcentcolon n\geq n_\delta\}_\delta,
    \end{equation}
    and
    \begin{equation}
        T_\delta \vcentcolon= \{x_n\vcentcolon n< n_\delta\}_\delta,
    \end{equation}
    which are respectively referred to as the \emph{nucleus} and the \emph{tail}. Now,
    \begin{align*}
        \left| \{x_n\vcentcolon n\in \N\}_\delta \right| & = \left| N_\delta\right| + \left| T_\delta \right|\\
        & = \left( x_{n_\delta} + 2\delta\right) + \left(n_\delta - 1\right)\cdot 2\delta\\
        & = x_{n_\delta} + 2\delta n_\delta.
    \end{align*}

    Since $x_{n_\delta} - x_{n_\delta + 1} \simeq 2\delta$, it is not hard to see that
    \begin{equation}
        n_\delta \simeq \delta^{-\frac{1}{\alpha + 1}}.
    \end{equation}
    Therefore,
    \begin{equation}
        \left| \{x_n\vcentcolon n\in \N\}_\delta \right| \simeq \delta^\frac{\alpha}{\alpha + 1},
    \end{equation}
    and consequently
    \begin{equation}
        \dimb \{x_n\vcentcolon n\in\N\} = \frac{1}{1+\alpha}.
    \end{equation}
\end{example}

\begin{figure}[h!]
    \centering
    \includegraphics[width=\linewidth]{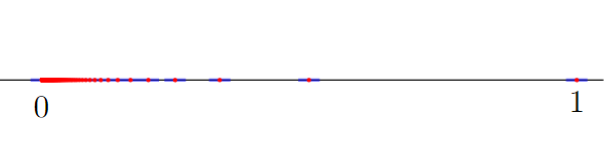}
    \caption{Visualisation of \autoref{exa:nuc_tail} for $\alpha = 1$ and $\delta = 0.02$, \cite{thesis_Crnkovic}}
    \label{fig:model_nuc_tail}
\end{figure}

\begin{remark}
    We will refer to the method from the previous example as the \emph{nucleus-tail method}. Notice that it is applicable whenever the distance between consecutive elements of the sequence in question eventually starts to decrease, and is therefore suitable for orbits of analytic maps converging to a fixed point. 
\end{remark}

\begin{proof}[Proof of \autoref{prop:dim_focus}]
    Due to \autoref{tm:flow_sec}, without loss of generality we may assume that $\Gamma_n$ are arcs of radii $x_n$ corresponding to $\theta \in[0,\pi]$. For $\delta>0$, the standard Nucleus-Tail argument for $(x_n)_n$ now gives
            \begin{equation}
                n_\delta \simeq \begin{cases}
                    \delta^\frac{1-k}{k}, & k\geq 2,\\
                    -\ln \delta, & k = 1,
                \end{cases}\quad \text{as }\delta \to 0,
            \end{equation}
            and
            \begin{equation}\label{focus-nuc}
                x_{n_\delta} \simeq \begin{cases}
                    \delta^\frac1k, & k\geq 2,\\
                    \delta, & k = 1,
                \end{cases}\quad \text{as } \delta \to 0.
            \end{equation}
            Indeed, in the case when $k\geq 2$, \autoref{tm:neveda} tells us that $x_{n}\simeq n^{-\frac{1}{k-1}}$, as $\delta \to 0$, and since $x_{n_\delta}\simeq \delta^\frac1k$, the statement follows. In the case when $k = 1$, that is when $P$ is hyperbolic, Nucleus-Tail argument again gives us
            \begin{equation*}
                x_{n_\delta}\simeq p^{-1}(\delta) \simeq \delta,\quad \text{as } \delta \to 0,
            \end{equation*}
            where $p(x)\vcentcolon= x - P(x) \sim c x$, as $x\to 0$. Since there are $\underline{c},\overline{c}\in (0,1)$ such that
            \begin{equation*}
                \underline{c} x \leq P(x) \leq \overline{c} x,\quad x\in[0, x_0],
            \end{equation*}
            we get
            \begin{equation*}
                \frac{\ln \frac{x_{n_\delta}}{x_0}}{\ln \underline{c}} \leq n_\delta \leq \frac{\ln \frac{x_{n_\delta}}{x_0}}{\ln \overline{c}},
            \end{equation*}
            so 
            \begin{equation*}
                n_\delta \simeq -\ln \delta,\quad \text{as } \delta \to 0.
            \end{equation*}
            
            Standardly, we have
            \begin{align*}
                \left| \left( \cup_{n\in\N_0} \Gamma_n\right)_\delta\right| & = \left| \left( \cup_{n\geq n_\delta} \Gamma_n\right)_\delta\right| + \sum_{n<n_\delta} \left|(\Gamma_n)_\delta\right|,
            \end{align*}
            and
            \begin{equation*}
                \frac{\pi}{2}\left( x_{n_\delta} + \delta\right)^2 \leq \left| \left( \cup_{n\geq n_\delta} \Gamma_n\right)_\delta\right| \leq \frac{\pi}{2}\left( x_{n_\delta} + \delta\right)^2 + 2\delta(x_{n_\delta}+\delta).
            \end{equation*}
            If $k = 1$, $\underline{c}^n x_0 \leq x_n \leq \overline{c}^n x_0,\, n\in\N$, and \eqref{focus-nuc} gives $\left| \left( \cup_{n\geq n_\delta} \Gamma_n\right)_\delta\right| = O(\delta^2)$. Now
            \begin{align*}
                \sum_{n<n_\delta} \left|(\Gamma_n)_\delta\right| & = \sum_{n<n_\delta} \left( 2\pi \delta x_n + \pi\delta^2\right)\\
                 & = \pi n_\delta\delta^2 + 2\pi \delta \sum_{n<n_\delta} x_n \simeq \delta,\quad \text{as }\delta \to 0.
            \end{align*}

            In the case $k \geq 2$, we have $\left| \left( \cup_{n\geq n_\delta} \Gamma_n\right)_\delta\right| \simeq (x_{n_\delta}+\delta)^2 \simeq \delta^\frac{2}{k}$, as $\delta\to 0$, and due to \cite[Theorem 1]{ezz07} we also have $x_n \simeq n^{-\frac{1}{k-1}}$, as $n\to\infty$. Again,
            \begin{align*}
                \sum_{n<n_\delta} \left|(\Gamma_n)_\delta\right| & = \sum_{n<n_\delta} \left( 2\pi \delta x_n + \pi \delta^2\right)\\
                 & = \pi n_\delta \delta^2 + 2\pi \delta \sum_{n<n_\delta} x_n.
            \end{align*}
            Since $n_\delta \delta^2 \simeq \delta^\frac{k+1}{k}$, as $\delta\to 0$, and
            \begin{align*}
                2\pi\delta \sum_{n<n_\delta} x_n  & \simeq \delta \sum_{n<n_\delta} \frac{1}{n^\frac{1}{k-1}}\\
                & \simeq \delta \int_1^{n_\delta} \frac{1}{x^\frac{1}{k-1}} dx\\
                & \simeq \begin{cases}
                    \delta n_\delta ^\frac{k-2}{k-1}, & k > 2,\\
                    \delta \ln n_\delta, & k = 2,
                \end{cases}\\
                & \simeq \begin{cases}
                    \delta^\frac{2}{k}, & k > 2,\\
                    \delta \left(-\ln \delta\right), & k = 2,
                \end{cases}\quad \text{as }\delta \to 0.
            \end{align*}
            The desired result directly follows.
\end{proof}

\begin{proof}[Proof of \autoref{prop:dim_para}]
    As in \autoref{ob:para}, without loss of generality we assume that $M \equiv 1$ in \eqref{eq:gen-parab}. Let $Y(t,x,y)$ now be the $y$-component of the flow of \eqref{eq:gen-parab}, starting at $(x,y)$. Since the change of coordinates $(x,y)\mapsto (x, Y(-x,x,y))$ is regular in a neighbourhood of the origin, i.e. bi-Lipschitz, it preserves the Minkowski dimension of sets. This change of coordinates transforms $(\Gamma_n)_n$ into parallel line segments $(S_n)_n$, with endpoints on the curve $\{y = Y(-x,x,0)\}$. Notice that $Y(-x,x,0) = \alpha x^2 + o(x^2)$, for some $\alpha \neq 0$. Now we separately consider the cases when $P$ is hyperbolic and parabolic. Let's first denote $\Tilde{Y}(x) = Y(-x, x, 0) = \alpha x^2 + o(x^2), \alpha \neq 0$.

    \begin{figure}[h!]
        \centering
        \includegraphics[width=\linewidth]{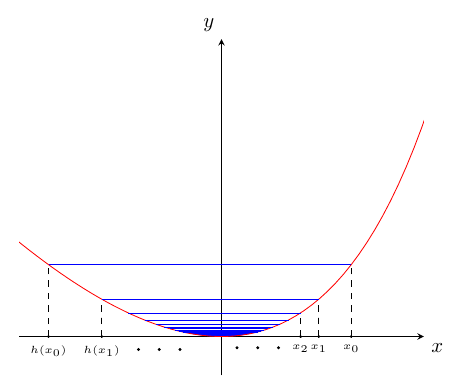}
        \caption{Family $(S_n)_n$ "after straightening" of the parabolic contact, \cite{thesis_Crnkovic}}
        \label{fig:post_st_chirp}
    \end{figure}

    \begin{figure}[h!]
        \includegraphics[width=0.5\linewidth]{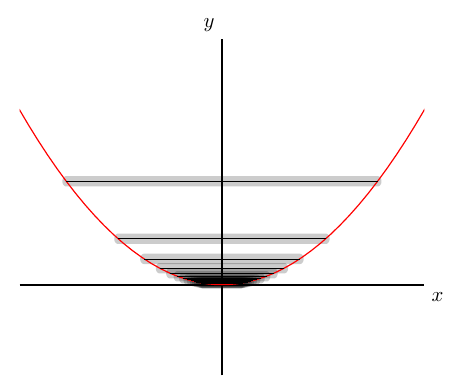}
        \includegraphics[width=0.5\linewidth]{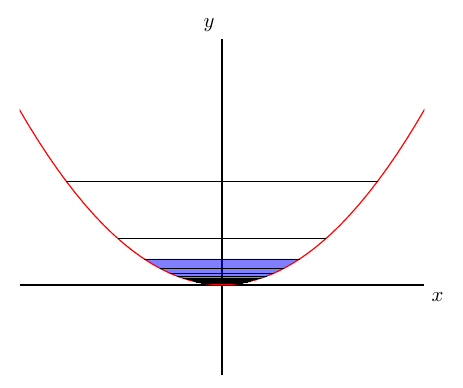}
        \caption{$\delta$-neighbourhood and the "nucleus" in the proof of \autoref{prop:dim_para}, \cite{thesis_Crnkovic}}
        \label{fig:para_analysis}
    \end{figure}
    
    \begin{enumerate}
        \item[\bf{case 1:} $k = 1$]
        Since $\Tilde{Y}'(x) = 2\alpha x + o(x)$, there is a small positive $r>0$ such that $\Tilde{Y}$ is strictly increasing on $(0,r)$. Therefore, there is an inverse $\Tilde{Y}^{-1}\vcentcolon\left(0, \Tilde{Y}(r)\right)\to \left(0, r\right)$.e now define $\Tilde{P}\vcentcolon= \Tilde{Y}\circ P\circ \Tilde{Y}^{-1}$. For the sequence $y_n\vcentcolon= \Tilde{Y}(x_n)$, we have
        \begin{equation}
            y_{n+1} = \Tilde{Y}(x_{n+1}) = \Tilde{Y}(P(x_n)) = \Tilde{Y}(P(\Tilde{Y}^{-1}(y_n))) = \Tilde{P}(y_n).
        \end{equation}
        Notice that $\Tilde{P}(y) \sim (1-c)^2y$ and $\Tilde{P}'(y) \sim (1-c)^2$, as $y\to 0$. Therefore, $\Tilde{P}$ and $\Tilde{p}(y) = y - \Tilde{P}(y)$ are strictly increasing for $r$ sufficiently small. Consequently, $(y_n - y_{n+1})_n$ is strictly decreasing. The nucleus-tail argument for the sequence $(y_n)_n$ now yields
        \begin{equation}
            y_{n_\delta} \simeq \delta, n_\delta \simeq -\ln \delta,\quad \text{as } \delta \to 0.
        \end{equation}
        Now,
        \begin{align*}
            \left| \left( \cup_{n\in\N_0} S_n\right)_\delta\right| = \left| \left( \cup_{n\geq n_\delta} S_n\right)_\delta\right| + \sum_{n<n_\delta} \left|(S_n)_\delta\right|.
        \end{align*}
        Since $\left| \left( \cup_{n\geq n_\delta} S_n\right)_\delta\right| = O(\delta^\frac32)$ and $\sum_{n<n_\delta} |(S_n)_\delta| \simeq \delta$, the desired results follows.

        \item[\bf{case 2:} $k\geq 2$]
        Notice that for $y_n \vcentcolon= \Tilde{Y}(x_n)$, we have
        \begin{align*}
            y_n - y_{n+1} & \sim 2\alpha x_n (x_n - x_{n+1})\\
            & \sim 2\alpha x_n cx_n^m\\
            & \sim 2c \sqrt{\frac{y_n}{a}}^{m+1}\\
            & \sim 2ca^{1-\frac{m+1}{2}}y_n^\frac{m+1}{2},\quad \text{as } n\to \infty.
        \end{align*}
        Nucleus-tail argument now gives us that $y_{n_\delta} \simeq \delta^\frac{2}{m+1}$, as $\delta\to 0$, and due to \autoref{tm:neveda} we know that $y_n\simeq n^{-\frac{1}{\frac{m+1}{2}-1}} \simeq n^{-\frac{2}{m-1}}$, as $n\to\infty$.
        
        Now,
        \begin{equation}
            \left|\left(\cup_{n\geq n_\delta} S_n\right)_\delta\right| \simeq y_{n_\delta}^\frac32 \simeq \delta^\frac{3}{m+1}.
        \end{equation}
        On the other hand
        \begin{align*}
            \sum_{n<n_\delta} |(S_n)_\delta| & \simeq \delta \cdot \sum_{n<n_\delta} x_n\\
            & \simeq \delta \cdot \sum_{n<n_\delta} n^{-\frac{1}{m-1}}\\
            & \simeq \begin{cases}
                \delta \cdot n_\delta^\frac{m-2}{m-1},& m>2,\\
                \delta \cdot \ln n_\delta, & m = 2,
            \end{cases}\\
            & \simeq \begin{cases}
                \delta^\frac{3}{m+1}, & m>2,\\
                -\delta \cdot \ln \delta, & m = 2,
            \end{cases} \quad \text{as }\delta \to 0.
        \end{align*}
    \end{enumerate}
   
\end{proof}

\section{The reconstructibility question}\label{sec:recon}    

    Since \autoref{thm:PPfrac} and \autoref{thm:FFfrac} provide a one-to-one correspondence between the order of (pseudo) foci and the fractal properties of the associated spiral trajectories, the realization question is reduced to the exploration of composition of various types of semi-monodromy and/or transition maps.
    
\subsection{Proofs of \autoref{thm:FF_real} and \autoref{thm:mixed_real}}
    
    Before we start the proofs of \autoref{thm:FF_real} and \autoref{thm:mixed_real}, we state and prove the following auxiliary lemma.

    \begin{lemma}\label{lm:diff_squares}
        Let $f(x) = -x + \sum_{n=2}^\infty a_nx^n$ and $g(x) = -x + \sum_{n = 2}^\infty b_n x^n$ be analytic at $0$. Then, for $k = min\{ n\in \N\vcentcolon a_n \neq b_n\}$, we have
        \begin{equation}
            f(f(x)) - g(g(x)) = \left( (-1)^k -1\right)(a_k-b_k)^k + o(x^k).
        \end{equation}
        
    \begin{proof}
        Direct calculation yields
            \begin{align*}
                f(f(x)) - g(g(x)) & = f(f(x)) - g(f(x)) + g(f(x)) - g(g(x))\\
                & = f(f(x)) - g(f(x)) + \left(-f(x) + \sum_{n=2}^\infty b_n f(x)^n\right)\\ & - \left( -g(x) + \sum_{n=2}^\infty b_n g(x)^n\right)\\
                & = (f-g)(f(x)) - (f-g)(x) + \sum_{n=2}^\infty b_n( f(x)^n - g(x)^n)\\
                & = (f-g)(f(x)) - (f-g)(x) + o(f(x)-g(x))\\
                & = (a_k-b_k)(-x)^k - (a_k-b_k)x^k + o(x^k).
            \end{align*}
    \end{proof}
    \end{lemma}

    \begin{proof}[Proof of \autoref{thm:FF_real}]
        We treat separately cases when $k_1 = k_2$ and when $k_1>k_2$.
            
            \begin{enumerate}
            
                \item Case $k_1 = k_2$:\\
                If $k\geq 2k_1+1$, the \autoref{prop:focus_semi_real} tells us that we can construct an analytic vector field $V_1$ with a focus at the origin such that the associated semi-monodromy is $m_1(x) = -x + x^{2k_1+1}+o(x^{k})$, and a vector field $\Tilde{V}_2$ such that for the associated semi-monodromy $\tilde{m}_2$ we have $\tilde{m}_2(x) = -x + x^{2k_1+1} - x^k + o(x^k)$. The vector field $V_2\vcentcolon= - \Tilde{V}_2$ is therefore analytic and the associated semi-monodomy maps $m_2$ satisfies $m_2^{-1}(x) = \tilde{m}_2(x) = -x+x^{2k_1+1}-x^k + o(x^k)$. Now, since $m_1(x) - m_2^{-1}(x) = x^k + o(x^k)$, we also know that $m_2(m_1(x)) = x - x^k + o(x^k)$, and the first part of the claim follows.

                If $k < 2k_1+1$ is even. Let $m_1$ be as in the first case. We can conjugate $m_1$ with $\varphi(x) = x - x^k$ to get
                \begin{equation*}
                    \varphi^{-1}\circ m_1 \circ \varphi(x) = - x + 2x^k + o(x^k).
                \end{equation*}
                \autoref{prop:focus_semi_real} now tells us that we can construct an analytic vector field $V_1'$ such that the associated semi-monodromy is $m_1'(x) = -x + 2x^k + o(x^k)$. Since $m_1'(m_1(x)) = x - 2x^k + o(x^k)$, the existential part of the claim follows.

                Let now $k<2k_1+1$ be odd, and assume that there are analytic vector fields $V_1$ and $V_2$ with a (weak) focus of order $k_1/k_2$ at the origin such that for the associated semi-monodromy maps $m_1$ and $m_2$, it holds that $m_2(m_1(x)) = x - ax^{k} + o(x^{k}), a\neq 0$, which is equivalent to $m_1(x) - m_2^{-1}(x) = ax^{k} + o(x^{k})$.
                We know that $m_1(m_1(x)) - m_2^{-1}(m_2^{-1}(x)) = O(x^{2k_1+1})$, but on the other hand, due to \autoref{lm:diff_squares} we have
                \begin{equation*}
                    m_1(m_1(x)) - m_2^{-1}(m_2^{-1}(x)) = -2ax^k + o(x^k) \neq O(x^{2k_1+1}),
                \end{equation*}
                and we arrive at a contradiction.

                \item Case $k_1>k_2$:\\
                If $k>2k_2+1$, assume that there are analytic vector fields $V_1$ and $V_2$ with a (weak) focus of order $k_1$ and $k_2$ respectively at the origin such that for the associated semi-monodromy maps $m_1$ and $m_2$ it holds that $m_2(m_1(x)) = x - ax^k + o(x^k), a\neq 0$.
                Now $\{m_1\}_{2k_2+1} = \{m_2^{-1}\}_{2k_2+1}$ and necessarily $m_1(m_1(x)) = m_2^{-1}(m_2^{-1}(x)) + o(x^{2k_2+1})$, which is in contradiction with the foci of $V_1$ and $V_2$ being of different orders.

                If $k = 2k_2+1$, we again, by \autoref{prop:focus_semi_real}, construct analytic vector fields $V_1$ and $V_2$ with a focus at the origin, such that the associated semi-monodromy maps are $m_1(x) = -x + x^{2k_1+1}+o(x^2k_1+1)$ and $m_2(x) = -x+x^{2k_2+1}+o(x^{2k_1+1})$ respectively. Now, $m_2(m_1(x)) = x - x^{2k_2+1} + o(x^{2k_2+1})$ and we have achieved what we wanted.

                If $k < 2k_2 + 1$ and $k$ is even, we again, by \autoref{prop:focus_semi_real}, construct analytic vector fields $V_1$ and $V_2$ with a focus at the origin, such that the associated semi-monodromy maps are $m_1(x) = -x + x^{2k_1+1}+o(x^2k_1+1)$ and $m_2(x) = -x+x^{2k_2+1}+o(x^{2k_1+1})$ respectively. We can conjugate $m_1$ with $\varphi(x) = x - x^k$ to get
                \begin{equation*}
                    \varphi^{-1}\circ m_1\circ\varphi(x) = -x + 2x^k + o(x^k).
                \end{equation*}
                Again, \autoref{prop:focus_semi_real} tells us that we can construct an analytic vector field $V_1'$ such that the associated semi-monodromy map is $m'_1 = -x + 2x^k + o(x^k)$. Since $m_2(m'_1(x)) = x - 2x^k + o(x^k)$, we have indeed constructed the pair of vector fields with the wanted properties.

                Finally, let us assume that for and odd $k<2k_2+1$, there are analytic vector fields $V_1$ and $V_2$ with foci of order $k_1$ and $k_2$ at the origin, such that, for the associated semi-monodromy maps $m_1$ and $m_2$, it holds that $m_2(m_1(x)) = x - ax^k + o(x^k)$, $a\neq 0$, which is equivalent to $m_1(x) - m_2^{-1}(x) = a^k + o(x^k)$. We know that $m_1(m_1(x))-m_2^{-1}(m_2^{-1}(x)) = O(x^{2k_2+1})$, but on the other hand, due to \autoref{lm:diff_squares}, we have
                \begin{equation*}
                    m_1(m_1(x)) - m_2^{-1}(m_2^{-1}(x)) = -2ax^k + o(x^k) \neq O(x^{2k_2+1}),
                \end{equation*}
                and we arrive at a contradiction.
                \end{enumerate}
        \end{proof}

        \autoref{thm:mixed_real} is proven similarly.

        \begin{proof}[Proof of \autoref{thm:mixed_real}]
            Let us first show that such vector fields can only exist if $n\leq 2k+1$.
            
            To this end, assume that there exist functions $m(x) = -x + O(x^2)$ and $h(x) = -x + O(x^2)$, analytic at the origin, such that $m(m(x)) = x + \alpha x^{2k+1} + o(x^{2k+1})$, $\alpha\neq 0$ and $h(h(x)) = x$. The function $m$ represents the semi-monodromy of a focus, and $h$ represents the transition function associated to a parabolic contact, or a semi-monodromy of a center. Additionally assume that $h(m(x)) = x + \beta x^n + o(x^n)$, $\beta \neq 0$. Similarly as before,
            \begin{equation*}
                m(x) - h(x) = m(x) - h^{-1}(x) = -\alpha x^n + o(x^n).
            \end{equation*}
            Since, $n>2k+1$, we necessarily have
            \begin{equation*}
                \{h(h(x))\}_{2k+1} = \{m(m(x))\}_{2k+1} = x + \alpha x^{2k+1},
            \end{equation*}
            which is a contradiction with the assumption that $h$ is an involution. Therefore, $n$ is indeed necessarily less than or equal to $2k+1$.

            Consider now the following cases.

            \begin{enumerate}
                \item {Case $n = 2k+1$:}

                By \autoref{prop:focus_semi_real}, we know that we can construct an analytic vector field $V$ with a focus at the origin, such that the associated semi-monodromy map $m$ satisfies $m(x) = -x + x^{2k+1} + o(x^{2k+1})$. By \autoref{prop:para_trans_char}, there also exist analytic vector fields with a parabolic contact at the origin and a center at the origin, such that their transition map and semi-monodromy map respectively is $h(x) = -x$. Since $h(m(x)) = x - x^{2k+1}+o(x^{2k+1})$, we have indeed constructed a wanted pseudo focus of order $n$.

                \item {Case $n<2k+1$ and $n$ is even:}

                 By \autoref{prop:focus_semi_real}, we know that we can construct an analytic vector field $V$ with a focus at the origin, such that the associated semi-monodromy map $m$ satisfies $m(x) = -x + x^{2k+1} + o(x^{2k+1})$. For $\phi(x) = x - x^n$, we define and involution $h(x) = \phi^{-1}(-\phi(x)) = -x+2x^n+o(x^n)$. By \autoref{prop:para_trans_char}, there also exist analytic vector fields with a parabolic contact at the origin and a center at the origin, such that their transition map and semi-monodromy map is respectively is $h$. Since $h(m(x)) = x + 2x^n + o(x^n)$, the obtained pseudo focus is indeed of order $n$.

                 \item {Case $n<2k+1$ and $n$ is odd:}

                 Assume that there exist functions $m(x) = -x + O(x^2)$ and $h(x) = -x + O(x^2)$ analytic at the origin such that $m(m(x)) = x + \alpha x^{2k+1} + o(x^{2k+1})$, $\alpha\neq 0$ and $h(h(x)) = x$. The function $m$ represents the semi-monodromy of a focus, and $h$ represents the transition function associated to a parabolic contact, or a semi-monodromy of a center. Additionally assume that $h(m(x)) = x + \beta x^n + o(x^n)$, $\beta \neq 0$. Therefore, $m(x) - h^{-1}(x) = -\beta x^m$, and, due to \autoref{lm:diff_squares}, we have
                 \begin{align*}
                     m(m(x)) - x & = m(m(x)) - h^{-1}(h^{-1}(x)) = 2\beta x^n + o(x^n),
                 \end{align*}
                 so we arrive at a contradiction.
            \end{enumerate}
        \end{proof}
   
\section{Conclusions and perspectives}\label{sec:conc}

    This manuscript provides a complete fractal treatment of pseudo-foci of types introduced in \cite{cgp01}. Notice that all our calculations only depend on the first non-zero term in the asymptotic expansion of the displacement maps associated to the first return map, so similar results are expected to hold for $C^\infty$ systems, with possible small alterations of some arguments. During consultations with the author's thesis defense committee, Prof. Pavao Mardešić suggested generalizations of these results in the spirit of the work of Nakai \cite{nakai94}. The first return maps we have studied are of the form $P = h^-\circ h^+$, where $h^\pm$ are the semi-monodromy maps in the upper and lower half plane respectively, and an orbit of a point can be understood as the action of a \emph{pseudogroup} $G = \{P^n \colon n\in \N_0\}$ on a point. The work of Nakai involves the study of generalized orbits under the action of general \emph{nonsolvable} pseudogroups of holomorphic diffeomorphisms in $(\mathbb{C}, 0)$. Motivated by this one can consider orbits under the action of the pseudogroup generated by $S = \{h^-\circ h^+, h^+\circ h^-, h^+\circ h^+, h^-\circ h^-\}$. These orbits can be thought of as sets of points of intersection of a transversal and a trajectory of a piecewise analytic system where the switch between the two subsystems does not necessarily happen every time we intersect the $x$-axis. This approach can even be enriched with a stochastic element, where one might study the \emph{expected} or \emph{a.s.} fractal properties of such orbits, under certain probability distributions that determine the switch between the two regimes.

\section*{Acknowledgements}\label{sec:ack}

    The contents of this paper arose from a chapter of the authors PhD thesis that is part of a joint degree (cotutelle de these) awarded by University of Zagreb, Croatia, and Hasselt University, Belgium. The author would like to thank his PhD advisor, Professor Renato Huzak, for suggesting the fractal analysis of pseudo foci as part of their research. The author would also like to thank Professor Maja Resman for suggesting the reconstructibility question.
    
    This research was supported by Croatian Science Foundation (HRZZ) grant IP-2022-10-9820, "Global and local dynamics on surfaces".

\bibliographystyle{alpha}
\bibliography{citation}

\end{document}